\documentclass[]{elsart}

\usepackage{amsmath}
\usepackage{mathdots}
\usepackage{amsfonts}
\usepackage{latexsym}
\usepackage{subfigure}

\newcommand{\bv}[1]{{\mbox{\boldmath $ #1$}}}
\def\vbar{\bv{\bar{v}}}
\def\abar{\bv{\bar{a}}}

\def\a{\bv{a}}
\def\f{\bv{f}}
\def\v{\bv{v}}

\def\m{\bv{m}}
\def\y{\bv{y}}
\def\x{\bv{x}}
\def\b{\bv{b}}

\def\X{\bv{X}}
\def\Y{\bv{Y}}

\def\ones{\bv{1}}
\def\cb{\bv{c}}
\def\db{\bv{d}}
\def\cbar{\bv{\bar{c}}}
\def\dbar{\bv{\bar{d}}}

\def\LToep{{\mathcal T}}

\def\at{{\tilde\a}}
\def\bt{{\tilde\b}}
\def\mt{{\tilde\m}}

\def\diag{{\rm diag}}
\def\Real{\mathbb{R}}  
\newcommand{\lbeq}[1]{{\label{OR:eq:#1}}}
\newcommand{\eq}[1]{{(\ref{OR:eq:#1})}}
\newcommand{\eqtwo}[2]{{(\ref{OR:eq:#1}, \ref{OR:eq:#2})}}
\newcommand{\be}[1]{\begin{equation} \lbeq{#1}}
\newcommand{\ee}{\end{equation}}
\newcommand{\Vector}[1]{{\left(\begin{matrix} #1 \end{matrix}\right)}}

\newtheorem{proposition}{Proposition}
\newtheorem{theorem}{Theorem}

\newtheorem{lemma}{Lemma}

\newenvironment{proof}{\textit{Proof:}\ }{$~\Box$}
\def\Tbb{{\mathbb L}}
\def\Rbb{{\mathbb R}}
\def\Nbb{{\mathbb N}}

\begin{document}
\begin{frontmatter}

\title{Resolution of the finite Markov moment problem\thanksref{titlefn}}

\thanks[titlefn]{Support by the Wolfgang Pauli Institute (Wien) in the frame of the
``Nanoscience" Programme and the European network HYKE, funded  by
the EC as contract HPRN-CT-2002-00282 is acknowledged.}

\author{Laurent Gosse}
\address{IAC--CNR ``Mauro Picone" (sezione di Bari)\\Via Amendola 122/D - 70126 Bari, Italy}
\ead{l.gosse@ba.iac.cnr.it}

\author{Olof Runborg\corauthref{cor}}
\corauth[cor]{Corresponding Author}
\address{NADA, KTH, 10044 Stockholm, Sweden}
\ead{olofr@nada.kth.se}
\date{}
\vskip 0.5cm
\begin{abstract}
We expose in full detail a constructive procedure to invert the
so--called ``finite Markov moment problem". The proofs rely on the
general theory of Toeplitz matrices together with the classical Newton's
relations. 
\end{abstract}
\begin{keyword} Inverse problems, Finite Markov's moment problem, Toeplitz matrices.
\MSC{65J22}
\end{keyword}
\end{frontmatter}

\section*{R\'esum\'e}

Nous pr\'esentons en d\'etail une proc\'edure constructive pour inverser le ``probleme fini des moments de Markov". Les preuves reposent sur la th\'eorie g\'en\'erale des matrices de Toeplitz et les classiques relations de Newton.

\section*{Version francaise abr\'eg\'ee}

Afin d'inverser le systeme fini et mal conditionn\'e \eq{mom-olof}, Koborov, Sklyar et Fardigola, \cite{kobo1,sklyar1} ont propos\'e un algorithme recursif non-lin\'eaire.
Dans \cite{GO} nous avons prouv\'e un lemme
le r\'eduisant a l'extraction de valeurs propres g\'en\'eralis\'ees, voir \eq{geneig}. Cette Note vise a expliquer en d\'etail les raisons pour lesquelles cette proc\'edure simplifi\'ee r\'esout le probleme des moments de Markov. Apres avoir rappel\'e quelques \'el\'ements de la th\'eorie des matrices de Toeplitz et les relations de Newton (Propositions \ref{toeplitz} et \ref{new}), nous reformulons cet algorithme simplifi\'e afin d'\'etablir facilement certains lemmes techniques. Finalement, le Th\'eoreme \ref{evenK} d\'emontre le lien entre valeurs propres g\'en\'eralis\'ees \eq{geneig} et l'inversion de \eq{mom-olof}.

\section{Introduction}

We aim at inverting a moment system often associated with the prestigious name of Markov, as appearing in \cite{poly,brecor,G,lewis,olof1} in several fields of application; consult \cite{diaco,krein,talenti} for general background on moment problems. 
The original problem is the following.
Given the {\em moments} $m_k$ for $k=1,\ldots,K$, find a bounded measurable
{\em density} function $f$ and a real value $X>0$ such that
\begin{enumerate}
\item $\int_0^X f(\xi)\xi^{k-1} d\xi=m_k, \qquad k=1,\ldots, K,$
\item $|f(\xi)|=1$ almost everywhere on $]0,X[$,
\item $f$ has no more than $K-1$ discontinuity points inside $]0,X[$.
\end{enumerate}
The solution is a piecewise constant function taking values in
$\{-1,1\}$ a.e. on $]0,X[$ and changing sign in at most $K-1$ points,
which we denote $\{u_k\}$, ordered such that 
$0\leq u_1 \leq \cdots\leq u_K=X$.
Finding $\{u_k\}$ from $\{m_k\}$ is
an ill-conditioned problem when the $u_k$ values come close to each 
other; its Jacobian is a Vandermonde matrix and iterative numerical 
resolution routines require extremely good starting guesses. 
For less than 4 moments, however, a direct method based on
solving polynomial equations was presented in \cite{olof1}. 
Here we are concerned with an arbitrary number of moments $K \in \Nbb$. 

We consider a slightly modified version of the problem 
which is more relevant for us. In this case
$f$ takes values in $\{1,0\}$ instead of $\{-1,1\}$
and the moments are scaled as $m_k \to km_k$. 
Moreover, to simplify the discussion we confine ourselves to 
the case when $K$ is even, setting $K=:2n$.
The resulting problem can then be written as an algebraic 
system of nonlinear
equations: Given $m_k$ find $u_k$ such that
\be{mom-olof}
  m_k = \sum_{j=1}^{n} u_{2j}^k-u_{2j-1}^k,
\qquad k=1,\ldots, K=2n.
\ee
An algorithm for solving this problem was
presented by Koborov, Sklyar and Fardigola
in \cite{kobo1,sklyar1}. It requires
solving a sequence of high degree polynomial equations,
constructed through a rather complicated process with unclear stability
properties.
In \cite{GO} we showed that the algorithm 
can be put in a more simple form that makes it
much more suitable for numerical implementation.
The simplified algorithm reads
\begin{enumerate}
\item Construct the matrices $A$ and $B$:
\be{ABdef}
   A = \Vector{1 & & & \\
-m_1 & 2 &\\
\vdots &\ddots &\ddots & \\
-m_{2n-1} & \hdots & -m_1 & 2n
},
\qquad
   B = \Vector{1 & & & \\
m_1 & 2 &\\
\vdots &\ddots &\ddots & \\
m_{2n-1} & \hdots & m_1 & 2n
}.
\ee
\item 
Let $\m = (m_1, m_2, \ldots, m_{2n})^T$ and solve
the lower triangular Toeplitz linear systems
\be{absyst}
A\a=\m, \qquad B\b=-\m,
\ee
to get $\a$ and $\b$.
\item Construct the matrices $A_1$, $A_2$
from $\a= (a_1, a_2, \ldots, a_{2n})^T$ 
as
$$ 
   A_1=\Vector{a_1 & a_2 & \hdots & a_{n} \\
a_2 & a_3 & \hdots & a_{n+1} \\
            \vdots & \vdots & \ddots & \vdots \\
a_{n} & a_{n+1} & \hdots & 
a_{2n-1}
}, \qquad
   A_2=\Vector{a_2 & a_3 & \hdots & a_{n+1} \\
a_3 & a_4 & \hdots & a_{n+2} \\
            \vdots & \vdots & \ddots & \vdots \\
a_{n+1} & a_{n+2} & \hdots & 
a_{2n}
},
$$
and the corresponding matrices $B_1$, $B_2$ from $\b$.
\item The values $\{u_k\}$ can then be computed as the generalized 
eigenvalues of the problems 
\be{geneig}
A_2\v = u A_1\v, \qquad 
B_2\v = u B_1\v.
\ee
\end{enumerate}
The forthcoming section is devoted to a complete justification of this
algorithm.
We recall that these inversion routines have been shown to be numerically
efficient in the paper \cite{GO}.


\section{Analysis of the algorithm}

We begin by stating two classical results of prime importance for the
analysis.

Let $\Tbb^n\subset\Real^{n\times n}$ be the set of lower triangular
$n\times n$ real Toeplitz matrices. 
We define the diagonal scaling matrix and the 
mapping $\LToep:\Real^n\to\Tbb^n$ as
$$
  \Lambda=
   \Vector{ 
 0 &&\\
  &  1 &\\
& &\ddots & \\
 &  & & n-1
}; \qquad
  \LToep(\x) :=
   \Vector{ 
 x_1 &&\\
 x_{2} &  x_1 &\\
\vdots &\ddots &\ddots & \\
 x_{n} & \hdots & x_{2} & x_1
},
\qquad \x=\Vector{x_1 \\ x_2 \\ \vdots \\ x_{n}}.
$$
The mapping $\LToep$ has 
the following properties, see {\it e.g.} \cite{bini}:
\begin{proposition}\label{toeplitz}
Lower triangular Toeplitz matrices commute and $\Tbb^n$ is closed
under matrix multiplication and (when the inverse exists) inversion,
\be{Tmprop1}
     \LToep(\x)\LToep(\y) = \LToep(\y)\LToep(\x) \in \Tbb^n,  \qquad \LToep(\x)^{-1}\in \Tbb^n.
\ee
Moreover, $\LToep$ is linear and 
\be{Tmprop2}
     \LToep(\x)\y = \LToep(\y)\x,
\qquad   \LToep(\x)\LToep(\y) = \LToep(\LToep(\x)\y).
\ee
The $\Lambda$ matrix has the property
\be{Tmprop4}
   \LToep(\Lambda\x) = \Lambda\LToep(\x)-\LToep(\x)\Lambda.
\ee
\end{proposition}

Another result that we rely heavily on is the classical Newton relations, 
see {\it e.g.} \cite{proofofNewton}:
\begin{proposition}[Newton's relations]\label{new}
Let $P$ be the $n$-degree polynomial,
$$
   P(x) = c_0 + c_1 x + \cdots + c_nx^n =: c_n(x -x_1)\cdots(x-x_n).
$$
Set $S_0 = n$ and define $S_k$ for $k>0$
as the sum of the roots of $P$ taken to the power $k$, $S_k = \sum_{j=1}^n x_j^k$. Then, the $n+1$ following relations hold:
\begin{align}
 c_{k}S_0 + c_{k+1}S_1 + \cdots + c_nS_{n-k}   &= kc_{k}, & &  k =0,\ldots, n. 
\label{newton1}
\end{align}
\end{proposition}


\subsection{Reformulation of the simplified algorithm}

We want to write the equation $A\a=\m$ using the mapping $\LToep$: 
hence we augment the $\m$ and $\a$-vectors with a zero and one element, 
respectively, to get 
$\mt=(0,\ \m)^T$ and $\at=(1,\ \a)^T$, both in $\Real^{K+1}$.
We observe that the $A$-matrix in \eq{ABdef} is the lower right ${K}\times {K}$ block of $\Lambda-\LToep(\mt)$. Therefore,
$$
     (\Lambda-\LToep(\mt))\at = -\mt + \Vector{0 \\ A\a}
  =\Vector{0 \\ A\a-\m}.
$$
Thus the equation $A\a=\m$ in \eq{absyst} is equivalent to
\be{adef}
   \LToep(\mt)\at= \Lambda\at.
\ee
By the same argument, $B\b=-\m$ in \eq{absyst} is equivalent to
\be{bdef}
   \LToep(\mt)\bt= -\Lambda\bt.
\ee
with $\tilde\b=(1,\;\b)^T$.

We can then directly also show that $\LToep(\at)$ and $\LToep(\bt)$ are in fact each other's inverses.
\begin{lemma}\label{TaTbinv}
 When $\at$ are $\bt$ given by \eqtwo{adef}{bdef}
then $\LToep(\at)\LToep(\bt)=I$.
\end{lemma}
\begin{proof}
 By Proposition \ref{toeplitz} and \eqtwo{adef}{bdef},
$$
\Lambda\LToep(\at)\bt =
 \LToep(\Lambda\at)\bt+
\LToep(\at)\Lambda\bt =
 \LToep(\LToep(\mt)\at)\bt+
\LToep(\at)\Lambda\bt =
 \LToep(\at)\left[\LToep(\mt)\bt+
\Lambda\bt\right]=0.
$$
Thus $\LToep(\at)\bt$ lies in the nullspace of $\Lambda$ which is spanned by the vector $\ones:=(1, 0, ..., 0)^T$. Moreover, $(\LToep(\at)\bt)_1=1$ since the first elements of $\at$ and $\bt$ are both one and we must in fact have $\LToep(\at)\bt=\ones$. The lemma then follows from \eq{Tmprop2} and the definition of $\LToep$. \end{proof}

\subsection{What lies beneath the algorithm}

To understand the workings of the algorithm we need to
introduce some new quantities and determine how they
relate to $\at$, $\bt$ and $\mt$.

Let us start with some notation: we set
$x_j = u_{2j-1}$ and $y_j = u_{2j}$ for $j=1, ..., n$. 
Furthermore, we introduce the sums
\be{XYdef}
   X_k = \sum_{j=1}^{n} x_j^k, \qquad
   Y_k = \sum_{j=1}^{n} y_j^k, \qquad k=1,2, \ldots, K=2n,
\ee
and define $X_0=Y_0=K$. In the even case, it then holds that
$$
   m_k = \sum_{j=1}^{n} x_j^k - \sum_{j=1}^{n} y_j^k = X_k - Y_k, \qquad
k=1, \ldots, K=2n.
$$
We also define the two polynomials
\be{cdef}
   p(x)=(x-x_1)\cdots(x-x_{n})=:c_0 +c_1x + \cdots + c_{n-1} x^{n-1}+c_{n} x^{n},
\ee
and
\be{ddef}
   q(x)=(x-y_1)\cdots(x-y_{n})=:d_0 +d_1x + \cdots + d_{{n}-1} x^{{n}-1}+d_{n} x^{n}.
\ee
We note here that by construction $c_{n}=d_{n}=1$.

By applying (\ref{newton1}) to $x^{n}p(x)$ with $k=0,\ldots,{K}$ we get
$$
   \Vector{ 
 X_0 &&\\
 X_{1} &  X_0 &\\
\vdots &\ddots &\ddots & \\
 X_{{K}} & \hdots & X_{1} & X_0
}
\Vector{c_{n} \\ c_{{n}-1} \\ \vdots \\ c_0 \\ 0 \\ \vdots \\ 0}
=
\Vector{ {K}c_{n}  \\ ({K}-1)c_{{n}-1}  \\ \vdots \\ (K-n)c_0 \\ 0 \\ \vdots \\ 0}
$$
The analogous system of equations holds also for $Y_k$ and $d_k$. We introduce now some shorthand notation to write these equations in a concise form.
First we set $\cbar = (c_{n}, \ldots, c_0)^T\in\Real^{n+1}$ and $\dbar = (d_{n}, \ldots, d_0)^T\in\Real^{n+1}$. We then construct the larger
vectors, padded with zeros:
$\cb=(\cbar,\ {\bf 0})^T$ and
$\db=(\dbar,\ {\bf 0})^T$, both in $\Real^{{K}+1}$.
Finally, we let $\X=(X_0,\ldots,X_K)^T$ and
$\Y=(Y_0,\ldots,Y_K)^T$.
Using $\LToep$ and $\Lambda$ we can state the systems of equations above as follows:
\be{cdXYrel}
  \LToep(\X)\cb = ({K}I-\Lambda)\cb, \qquad
  \LToep(\Y)\db = ({K}I-\Lambda)\db.
\ee
We also clearly have $\mt=\X-\Y$.

Before we can relate $\cb$ and $\db$ with $\at$ we need the
following lemma:
\begin{lemma}\label{finv}
  Let $\f:\Real^n\to\Real^n$ be defined by $\f(\x) := \LToep(\x)^{-1}\Lambda\x$
for $\x$ with a non-zero first element. Then $\f(\x_1)=\f(\x_2)$ implies that $\x_1=\alpha \x_2$ for some non-zero $\alpha\in\Real$.
\end{lemma}
\begin{proof}
Suppose $\f(\x)=\y$. Then $\LToep(\x)\y=\Lambda\x$ and by Proposition \ref{toeplitz}, $(\LToep(\y)-\Lambda)\x=0$.
Hence $\f(\x_1)=\f(\x_2)$ implies that $\x_1$ and $\x_2$ both lie in the nullspace of $\LToep(\y)-\Lambda$. 
Since the top left element of $\Lambda$ is zero, it follows that 
the first element of $\y$ is zero 
and therefore the diagonal of $\LToep(\y)$ is zero. Consequently, the nullspace of $\LToep(\y)-\Lambda$ has the same dimension as that of $\Lambda$, which is one. \end{proof}

We can now merge together and express the general structure from
\eqtwo{cdef}{ddef} and \eqtwo{adef}{bdef} in the most concise way.
\begin{lemma}\label{cad}
Suppose $\cb$, $\db$ are defined by \eqtwo{cdef}{ddef} and $\at$, $\bt$ by \eqtwo{adef}{bdef}. Then
$$
   \LToep(\at)\cb =\db, \qquad
   \LToep(\bt)\db  =\cb.
$$
\end{lemma}
\begin{proof}
We only need to prove the left equality. The right one follows immediately
from Lemma \ref{TaTbinv}. Let $\v=\LToep(\at)\cb=\LToep(\cb)\at$. We want to show that $\v=\db$.
We note first that by \eq{cdXYrel}
$$
   \LToep(\X)\cb = ({K}I-\Lambda)\cb \quad\Rightarrow\quad
   \LToep(\cb)\X = ({K}I-\Lambda)\cb \quad\Rightarrow\quad
   \X = {K}\LToep(\cb)^{-1}\cb-\LToep(\cb)^{-1}\Lambda\cb,
$$
where $\LToep(\cb)$ is invertible since $c_{n}=1$.
Moreover, it is clear that $\LToep(\y)\ones = \y$
for all $\y$.
Hence, $\X = {K}\ones-\LToep(\cb)^{-1}\Lambda\cb$. In the same
way we also obtain $\Y = {K}\ones-\LToep(\db)^{-1}\Lambda\db$.
Then, 
\begin{align*}
  \LToep(\mt)\at &=
  \LToep(\at)\mt=
  \LToep(\at)(\X-\Y)\\
&= 
-\LToep(\at)\LToep(\cb)^{-1}\Lambda\cb
+\LToep(\at)\LToep(\db)^{-1}\Lambda\db.
\end{align*}
We now note that by Proposition \ref{toeplitz},
\begin{align*}
\LToep(\at)\LToep(\cb)^{-1}\Lambda\cb &=
\LToep(\cb)^{-1}\LToep(\at)\Lambda\cb =
\LToep(\cb)^{-1}\LToep(\Lambda\cb)\at = 
\LToep(\cb)^{-1}\Lambda\LToep(\cb)\at - \Lambda\at 
\\
&= 
\LToep(\cb)^{-1}\Lambda\v - \Lambda\at.
\end{align*}
Since also, $\LToep(\cb)\LToep(\at)=\LToep(\v)$ we get
\begin{align*}   
  \LToep(\mt)\at &=
-\LToep(\cb)^{-1}\Lambda\v + \Lambda\at
+\LToep(\at)\LToep(\db)^{-1}\Lambda\db
\\
 &= \LToep(\at)\Bigl[\LToep(\db)^{-1}\Lambda\db
-\LToep(\v)^{-1}\Lambda\v\Bigr] + \Lambda\at.
\end{align*}
Consequently, by \eq{adef}, $\LToep(\db)^{-1}\Lambda\db =\LToep(\v)^{-1}\Lambda\v$ and by Lemma \ref{finv}, $\v = \alpha \db$;
for some $\alpha\in\Real$. But for the first element in $\v$
we then have $v_1 = c_{n}=\alpha d_{n}$ and we get
$\alpha=1$ since $c_n=d_n=1$.
\end{proof}

Finally, we also establish the following lemma.
\begin{lemma}\label{Tcstruct}
Let $V$ and $W$ be the Vandermonde matrices corresponding to
the roots of $\tilde{p}(x):=xp(x)$ and $\tilde{q}(x):=xq(x)$ respectively. Then
$$
   V^TR\LToep(\cbar)V=\diag(\{\tilde{p}'(x_k)\}), \qquad
   W^TR\LToep(\dbar)W=\diag(\{\tilde{q}'(x_k)\}),
$$
where $R=\{\delta_{n+2-i-j}\}\in\Real^{n+1\times n+1}$ is the reversion 
matrix and $x_0=y_0=0$.
\end{lemma}
\begin{proof}
We have
$$
   (V^TR\LToep(\cbar)V)_{ij} = \x_{i-1}^TR\LToep(\cbar)\x_{j-1} =
\sum_{r =0}^n\sum_{\ell=0}^rc_r x_{i-1}^\ell  x_{j-1}^{r-\ell} =
\begin{cases}
 \frac{\tilde{p}(x_{i-1})-\tilde{p}(x_{j-1})}{x_{i-1}-x_{j-1}}, & i\neq j,\\
 \tilde{p}'(x_{i-1}), & i=j,
\end{cases}
$$
showing the left equality. The right equality follows in the same way.
\end{proof}

\subsection{Conclusion}

We can now conclude and show that the unknown values $u_j$ 
in \eq{mom-olof} are indeed the generalized eigenvalues of \eq{geneig}. 
\begin{theorem}\label{evenK}
  Suppose $K=2n$; let $\a$, $\b$ be defined by \eq{absyst}. If all values $\{x_j\} \cup \{y_j\}$
  are distinct, then
  $\{x_j\}$, $\{y_j\}$ are the generalized eigenvalues of \eq{geneig}.
\end{theorem}
\begin{proof}
Let $\at$, $\bt$ be defined by \eqtwo{adef}{bdef}, which is equivalent to \eq{absyst}. Also define $\cb$ and $\db$ as before by \eqtwo{cdef}{ddef}. By Lemma \ref{cad} we have $\LToep(\at)\cb=\db \in \Rbb^{2n+1}$, {\it i.e.}
\be{astruct}
\left(
  \begin{array}{llll|lllll}
         1 &             &        &    &&&&\\
         a_1 & 1         &        &    &&&&\\
         \vdots & \ddots & \ddots &    &&&&\\
         a_{n}  & \hdots & a_1  & 1    &&&&\\
\hline
a_{n+1} & a_{n}   &\hdots   & a_1 &         1 &             &   \\
a_{n+2} & a_{n+1} &\hdots   & a_2 &       a_1 & 1         &        &   \\
\vdots  & \vdots  &\ddots  & \vdots & \vdots         & \ddots & \ddots &   \\
a_{2n}  & a_{2n-1} &\hdots & a_{n}  &   a_{n-1}& \hdots & a_1  & 1 \\
 \end{array}
\right)
\Vector{ \\ \cbar \\  \\ \hline \\ { 0} \\ \ \\}
=
\Vector{ \\ \dbar \\  \\ \hline \\ { 0} \\ \ \\}.
\ee
Clearly, the lower left block of the matrix multiplied by $\cbar$ is zero, {\it i.e.} 
$\sum_{i=0}^n c_{i}a_{i+k}=0$ for $k=1, ..., n$. Now, let $v_i$ be the coefficients of the 
polynomial $v(x):=p(x)/(x-x_j)$ for some fixed $j$. Hence, by the special structure of \eq{cdef},
$$
   c_0+c_{1}x + ... + c_{n}x^n=: (v_1+v_2x+ ... + v_n x^{n-1})(x-x_j),
$$
and, for $i=0, ...,n$,
$$
c_{i}=\left\{\begin{array}{lr}-x_jv_{i+1}, & i=0, \\
v_i-x_jv_{i+1}, & 1\leq i \leq n-1, \\
v_i, & i=n.
\end{array}\right.
$$
Thus we deduce,
$$
0=-x_jv_1a_k+\sum_{i=1}^{n-1}(v_i-x_jv_{i+1})a_{i+k}+v_na_{n+k}=\sum_{i=1}^nv_i a_{i+k}-x_j\sum_{i=1}^nv_i a_{i+k-1},
$$
which is the componentwise statement of $A_2\v = x_j A_1\v$. It remains
to show that the rightmost sum is non-zero for at least some $k$,
so that $x_j$ is indeed a well-defined generalized eigenvalue.
Let $\abar=(1,a_1,\ldots,a_n)^T$ and $\vbar=(0,v_1,\ldots,v_n)^T$.
Then \eq{astruct} gives $\LToep(\abar)\cbar=\dbar$ and, using 
Lemma \ref{Tcstruct} while taking $k=1$ we have the sum
$$
 \sum_{i=1}^nv_i a_{i} = \abar^T\vbar
= (\LToep(\cbar)^{-1}\dbar)^T\vbar
= (V^TR\dbar)^T\diag(\{\tilde{p}'(x_k)^{-1}\})V^T\vbar=q(x_j),
$$
since $V^TR\dbar=\{q(x_k)\}$ and
$V^T\vbar=\{x_kv(x_k)\}=\{\delta_{k-j} \tilde{p}'(x_k)\}$. 
Hence, the sum is non-zero 
because $x_j\neq y_i$ for all $i,j$.
The same argument can be used for any $j$, 
which proves the theorem for $\{x_j\}$. The proof for $\{y_j\}$ 
is identical upon exchanging the 
roles of $\cb$, $\at$ and $\db$, $\bt$. This leads to \eq{geneig}. 
\end{proof}

\end{document}